\documentclass[12pt]{amsart}
\usepackage{mathpazo}
\usepackage{hyperref}
\usepackage{cleveref}
\usepackage{bm}
\usepackage[margin=1in]{geometry}

\usepackage{graphicx}
\usepackage[cmtip, all]{xy}
\usepackage{array}

\usepackage{amssymb,amsmath,latexsym,amsthm}

\DeclareMathOperator{\Gal}{Gal}
\DeclareMathOperator{\Frob}{Frob}

\DeclareMathOperator{\Aut}{Aut}
\DeclareMathOperator{\Jac}{Jac}
\DeclareMathOperator{\Sp}{Sp}
\DeclareMathOperator{\Conf}{Conf}

\newcommand{\field}[1]{\mathbb{#1}}

\newcommand{\Q}{\field{Q}}

\newcommand{\Z}{\field{Z}}
\newcommand{\F}{\field{F}}

\newcommand{\R}{\field{R}}

\newcommand{\A}{\field{A}}
\renewcommand{\P}{\field{P}}
\newcommand{\ord}{\mbox{ord}}

\newcommand{\ra}{\rightarrow}

\newcommand{\Tr}{\mbox{\textnormal Tr}}

\renewcommand{\a}{\langle a \rangle}

\newcommand{\beq}{\begin{displaymath}}
\newcommand{\eeq}{\end{displaymath}}
\newcommand{\beqn}{\begin{equation}}
\newcommand{\eeqn}{\end{equation}}
\newcommand{\Fqbar}{\overline{\field{F}}_q}
\newcommand{\Hn}{\mathsf{Hn}}

\DeclareMathOperator{\rank}{rank}

\numberwithin{equation}{section}
\newtheorem{theorem}[equation]{Theorem}
\newtheorem{remark}[equation]{Remark}
\newtheorem{lemma}[equation]{Lemma}

\newtheorem{corollary}[equation]{Corollary}
\newtheorem{proposition}[equation]{Proposition}

\usepackage[usenames,dvipsnames]{color}

\title{Nonvanishing of hyperelliptic zeta functions over finite fields}
\author{Jordan S. Ellenberg, Wanlin Li, Mark Shusterman}
\date{}

\begin{document}
	
\maketitle

\begin{abstract}
	Fixing $t \in \R$ and a finite field $\F_q$ of odd characteristic,
	we give an explicit upper bound on the proportion of genus $g$ hyperelliptic curves over $\F_q$ whose zeta function vanishes at $\frac{1}{2} + it$. 
	Our upper bound is independent of $g$ and tends to $0$ as $q$ grows.
	
\end{abstract}

\section{Introduction}
Let $p$ be an odd prime, set $q = p^k$ for some positive integer $k$, and denote by $\F_q$ the finite field with $q$ elements.
To (the smooth completion of) any hyperelliptic curve $C$ over $\F_q$ one associates a zeta function $Z_C(s)$.
Weil has shown that $Z_C(s) = 0$ implies that $s = \frac{1}{2} + it$ for some $t \in \mathbb{R}$.

It is widely believed that for any fixed $s = \frac{1}{2} + it$, the `vast majority' of (hyperelliptic) curves do not have $s$ as a zero of their zeta function.
For example, it follows from the work \cite{Chav97} of Chavdarov (and its improvement by Kowalski \cite{Kow06}) that for any fixed (large enough) $g$, the proportion of genus $g$ hyperelliptic zeta functions vanishing at $s$ tends to $0$ as $q \to \infty$.

Here we are concerned with the growing $g$ regime. Namely, for fixed $q$ (and $s$), we give an upper bound on
\begin{equation}
h_{q,s} := \sup_{g} \frac{\big|\{C \in \mathcal{H}_g(\F_q) : Z_C(s) = 0\}\big|}{\big|\mathcal{H}_g(\F_q)\big|}    
\end{equation}

where $\mathcal{H}_g(\F_q)$ is the family of genus $g$ hyperelliptic curves over $\F_q$. 
Our bound is better once $q$ is large, as given by our main result.

\begin{theorem}[Theorem \ref{ThmNonvanishing}] \label{FirstRes}
	
	Fix a prime $p$, a real number $t$, and set $s = \frac{1}{2} + it$. 
	Then as $k \to \infty$ we have
	\begin{equation}
	h_{p^k, s} \ll p^{-k/276}.
	\end{equation}
	In particular, $h_{p^k,s}$ tends to $0$ as $k$ tends to $\infty$.
\end{theorem}

This complements (but does not quite match) lower bounds on $h_{q,s}$ obtained by Li in \cite{Wanlin}.

Restricting $q$ to powers of a fixed prime $p$ is not always necessary. 
In case $s \neq \frac{1}{2}$, one can show (see \cite{MO}) using transcendental number theory (six exponentials theorem \cite[Chapter 2, Section 1]{Lang66}) that there are only finitely many $p$ for which $p^{-s}$ is algebraic, so $h_{q,s} = 0$ for any $q$ not divisible by these $p$ (as $Z_C(s)$ is a rational function in $q^{-s}$). Hence, it suffices to work with one characteristic at a time, as we do in the theorem above. For $s = \frac{1}{2}$, since the upper bound in Corollary \ref{MainCoro} holds for any $\ell$ when $q$ is a perfect square, we can conclude $\lim_{q \to \infty} h_{q, s} = 0$ ranging over $q$ which is an even power of a prime.

Additional motivation for Theorem \ref{FirstRes} comes from the ability to write $Z_C(s)$ as a rational function in $q^{-s}$, with the numerator being a quadratic Dirichlet $L$-function. 
Interpreted in this language of Dirichlet characters, Theorem \ref{FirstRes} improves (for all sufficiently large $q$) upon \cite[Corollary 2.1]{BF} of Bui and Florea (they give a lower bound of more than $94.27\%$ nonvanishing at $s = \frac{1}{2}$). Regarding the analogous vanishing problem for quadratic Dirichlet $L$-functions over $\mathbb{Z}$, 
we refer to the work \cite{Sound00} of Soundararajan and references therein\footnote{Results in \cite{BF} and \cite{Sound00} were stated at point $s=1/2$ but the methods can be extended to prove the statement for any point on the critical line.}.

As we explain in the last section, our theorem can be rephrased as an upper bound for the number of quadratic twists of a constant abelian variety which have positive rank.

\begin{corollary}[Corollary \ref{rank0}]\label{rank0intro}
	Let $A$ be a constant abelian variety defined over $\mathbb{F}_q(x)$. For each $f \in \mathbb{F}_{q^m}[x]$, denote by $A_f$ the quadratic twist of $A \otimes \mathbb{F}_{q^m}(x)$ by $f$.
	Let $R_{n,m}$ be the set $\{ f \in \mathbb{F}_{q^m}[x], \text{ squarefree, of } \deg n: A_f \text{ has positive rank} \}$.
	Then, $$\lim\limits_{m \to \infty} \limsup_{n \to \infty} \frac{|R_{n,m}|}{q^{m(n+1)}} =0 .$$
\end{corollary}

Motivated by \cite[Corollary 2.2]{BF} and the analogous results over $\mathbb{Z}$ of Conrey, Ghosh and Gonek from \cite{CGG98}, we bound the multiplicity of the zeros of $Z_C$, and obtain further information on nonvanishing at $s = \frac{1}{2}$.  

\begin{theorem}\label{SpecialCasefromMatrix}
	
	Let $C$ be a hyperelliptic curve of genus at least $2$ over $\F_q$ and $S$ be the set of Weierstrass points of $C$. The Frobenius acts on $S$ by permuting the $2g+2$ Weierstrass points via some permutation $\pi$.  Suppose that either
	\begin{itemize}
		\item $g$ is even and $\pi$ is a $(2g+2)$-cycle; or
		\item $\pi$ is the product of two disjoint cycles of odd length.
	\end{itemize}
	
	Then:
	
	\begin{enumerate}
		\item The point $s = \frac{1}{2}$ is not a zero of $Z_C$.
		\item All zeros of $Z_C$ are of multiplicity at most $2$. Moreover, if $\pi$ is the product of two disjoint cycles of {\em coprime} lengths, all zeros of $Z_C$ are simple.
	\end{enumerate}
	
\end{theorem}

In the language of Dirichlet characters, this implies in particular the nonvanishing (at the central point) in the case of prime conductor of degree not divisible by $4$ and therefore gives an explicit set of size on order $X/\log X$ of Dirichlet characters of conductor at most $X$ which have $L$-functions nonvanishing at the critical point. See the statement below.

\begin{corollary}\label{quadcharacterCorollary}
Let $\chi$ be a quadratic character over $\mathbb{F}_q(x)$ with conductor $f \in \mathbb{F}_q[x]$. If $f$ is irreducible and $4 \nmid \deg f$, then $L(1/2,\chi) \ne 0$. 
\end{corollary}

%(When the degree of the conductor is odd, one Weierstrass point is at $\infty$ and $\pi$ consists of a fixed point and a cycle of length $2g+1$.)  Thus by the prime number theorem for function fields, there is an explicit set of size on order $X/\log X$ of Dirichlet characters of conductor at most $X$ which have $L$-functions nonvanishing at the critical point; 

In particular, the number of quadratic characters with irreducible conductor of size at most $X$ whose $L$-function does not vanish at $s=1/2$ is $\gg X/\log X$ as $X \to \infty$.
This result improves on \cite[Corollary 2.6]{AK13} of Andrade and Keating and on \cite[Corollary 2.8]{ABJ16} of Andrade, Bae, and Jung, which give a proportion on order $(\log X)^{-2}$, and goes beyond the methods of \cite{AB18} by Andrade and Baluyot. For the analogous problem over $\mathbb{Z}$, we refer to the recent work \cite{BP18} of Baluyot and Pratt.

In fact, there is nothing special about hyperelliptic curves in Theorem~\ref{SpecialCasefromMatrix}.  A similar ``genus-theory" argument allows us to handle the case of cyclic $\ell$-covers of $\P^1$ for an odd prime $\ell$.

\begin{theorem}\label{ModpVersion}
	Let $\ell$ be an odd prime different from the characteristic of $\mathbb{F}_q$. Let $C$ be a smooth projective curve over $\mathbb{F}_q$ which admits a degree $\ell$ map $f: C \to \P^1_{\F_q}$ such that $f$ is Galois over $\mathbb{F}_q$ with Galois group isomorphic to $\Z/\ell\Z$. Let $S  \subset \P^1(\Fqbar)$ be the set of branch points of $f$. Let $\pi$ be the permutation induced by the Frobenius action on $S$, and suppose that $\pi$ is the composition of disjoint cycles of orders $k_1,k_2,\ldots, k_r$, all prime to $\ell$.  
	
	\begin{enumerate}
		\item Suppose the $k_i$ are mutually coprime and $r \leq 3$.  Then every zero of $Z_C$ has degree $\ell-1$.
		\item Define $\kappa_i$ to be $k_i$ if $k_i$ is odd and $k_i/2$ if $k_i$ is even.  Suppose that either
		\begin{itemize}
			\item $q$ is congruent to $1$ modulo $\ell$ and $r=2$, with both cycles of odd length; or
			\item There is no $i$ such that $q^{\kappa_i}$  is congruent to $1$ modulo $\ell$.
		\end{itemize}
		Then the point $s = \frac{1}{2}$ is not a zero of $Z_C$. 
	\end{enumerate}
\end{theorem}

We remark that this theorem, like Theorem~\ref{SpecialCasefromMatrix} above, can be used to produce a set of size on order $X/\log(X)^a$ of order-$\ell$ Dirichlet characters of conductor at most $X$ whose zeta functions are non-vanishing at $s=1/2$, for some power $a \in (0,1]$. See Corollary \ref{characterCorollary} for example.
This lower bound improves, for $\ell=3$, upon Corollary 1.3 of recent work by David--Florea--Lalin \cite{davidflorealalin} which gives a lower bound of the form $X^{1 - \epsilon}$ (for any $\epsilon > 0$).

\begin{corollary} \label{characterCorollary}
Let $\ell$ be an odd prime different from the characteristic of $\mathbb{F}_q$. Let $d$ be the order of $(q \bmod \ell)$ in $(\mathbb{Z}/\ell\mathbb{Z})^*$. Let $N(X)$ be the number of primitive degree $\ell$ Dirichlet characters $\chi_f:( \mathbb{F}_q[t]/f)^* \to \mathbb{C}^*$ with conductor $f$ satisfying $ q^{\deg f} \le X$ and $L(1/2,\chi_f) \ne 0$. We show the following:
\begin{itemize}
    \item if $d=1$, then $N(X) \gg \frac{X}{\log X}$ for $X \to \infty$;
    \item if $d$ is even, then $N(X) \gg \frac{X}{(\log X)^{1-\frac{\ell-1}{2d}}}$ for $X \to \infty$. In particular, since $d \mid \ell-1$, we have $N(X) \gg \frac{X}{(\log X)^{1/2}}$ for $X \to \infty$.
\end{itemize}
\end{corollary}

We do not get a lower bound of this quality in the case where $q$ has odd order modulo $\ell$.

The main idea that connects all the theorems in this paper is the study of $L$-functions modulo $\ell$.  The value of an $L$-function over $\F_q(x)$ at a complex number $s$ can be expressed as a polynomial $P(T) \in \Z[T]$ evaluated at $T = q^{-s}$.  So if we want to prove that $P(T)$ is nonvanishing, it suffices to prove that $P(T)$ is nonvanishing modulo $\ell$ for some prime $\ell$.  For Theorem~\ref{FirstRes}, we will show that, for suitably chosen $\ell$, the vanishing mod $\ell$ of the $L$-function is related to the dimension of a certain Frobenius eigenspace in the $\ell$-torsion of a hyperelliptic Jacobian over $\F_q$; the average size of this eigenspace can then be controlled by a point count on a moduli space over a finite field, 
which is a modest generalization and explication of the arguments in \cite{evw} and \cite{lt} respectively. 
For Theorem~\ref{SpecialCasefromMatrix}, on the other hand, we argue that under the given condition on Weierstrass points the $L$-function of $\chi_f$ is nonvanishing mod $2$ at $s=1/2$. For the similar Theorem~\ref{ModpVersion}, the $\ell$ is again the order of the Dirichlet character in question.

\subsection*{Acknowledgments} The authors are grateful to Chantal David, Zeev Rudnick, Alexandra Florea, and Emmanuel Kowalski for helpful comments and suggestions.  The first author was partially supported by NSF grant DMS-1700885 and by a fellowship from the Simons Foundation. We thank the referees for the valuable feedback and comments.

\section{Main Theorem and Proof}

\subsection{Setup and Notations}\label{Notation}

Throughout the paper, $\F_q$ is a finite field of odd characteristic $p$. Let $Q_{n,q}$ be the set of squarefree polynomials over $\F_q$ of degree $n$. For each $f \in Q_{n,q}$, write $J_f$ for the Jacobian of the hyperelliptic curve $$y^2=f(x)$$ and $P_f(x) \in \Z[x]$ for the characteristic polynomial of geometric Frobenius acting on the $\ell$-adic Tate module of $J_f$.  Let $\ell$ be a prime not equal to the characteristic of $\F_q$ and let $a$ be an element of $(\Z/\ell\Z)^*$. The elements $R$ of $J_f[\ell](\Fqbar)$ which satisfy
\begin{equation*} \label{FrobActEq}
\Frob_q \cdot R = aR.
\end{equation*}
form a finite-dimensional vector space over $\Z/\ell \Z$ and we denote by $m_a(f)$ the number of nonzero elements of this vector space. Note that $m_1(f)$ is just the number of $\F_q$-rational nontrivial $\ell$-torsion points of $J_f$. Let $Q_{n,q}^{a,\ell}$ be the set of squarefree polynomials $f$ over $\F_q$ of degree $n$ such that $m_a(f)$ is greater than $0$. 

Let $\alpha$ be a $q$-Weil number of weight $1$ with minimal polynomial $g_{\alpha}(x) \in \Z[x]$.
Namely, it is an algebraic integer whose absolute values under all complex embeddings equal $\sqrt{q}$. Let $Q_{n,q}^{\alpha}$ be the subset of $Q_{n,q}$ defined by $\{ f \in Q_{n,q} \mid P_f(\alpha^{-1}) = 0  \}$.  With notation introduced as above, if $g_{\alpha}(a) = 0 \bmod \ell$, then $|Q_{n,q}^{\alpha}| \le |Q_{n,q}^{a,\ell}|$.

\subsection{Rational points on twisted Hurwitz spaces over finite fields}

Our main tool will be the following result about the average size of the subspace of $\Jac(C)[\ell](\Fqbar)$ on which Frobenius acts by some specified scalar $a$, as $C$ ranges over hyperelliptic curves over $\F_q$.  More precisely, we study the variation as we range over $y^2 = f(x)$ with $f$ ranging over squarefree polynomials in $\F_q[x]$; this amounts to the same, since each isomorphism class of hyperelliptic curves is represented in this form the same number of times (assuming, of course, that the isomorphism classes are weighted inversely to the number of automorphisms they possess.)

\begin{proposition}
Let $a \in \left( \Z/\ell \Z \right)^\times$. With notation as in Section \ref{Notation}, there exist constants $C_\ell,N_\ell,Q_\ell$ only depending on $\ell$ such that 
\beq
\left\lvert \frac{\sum_{f \in Q_{n,q}} m_a(f)}{|Q_{n,q}|} -1 \right\rvert \leq C_\ell q^{-1/2}
\eeq
for all $n \geq N_\ell$ and $q \geq Q_\ell$.
\label{prop:twistedhurwitz}
\end{proposition}

\begin{remark} 
\normalfont While this paper was in proof, we learned that Proposition~\ref{prop:twistedhurwitz} follows from the proof of Theorem 1.1 of \cite{lt}, which is in fact more general, and the arguments used are essentially the same as those here.  We have left the proof of Proposition~\ref{prop:twistedhurwitz} in the present paper because the form in which we present the proof here is conducive to proving the explicit bounds for the stable range obtained in Proposition~\ref{prop:explicitbounds} below.  It would be interesting to address the questions of explicit bounds for the stable range in the more general situations considered by \cite{lt}, where the group-theoretic part of the proof of Proposition~\ref{prop:explicitbounds} would presumably be more complicated.
\end{remark}

When $a=1$ and $n$ is odd, Proposition~\ref{prop:twistedhurwitz} is essentially Theorem 8.8 of \cite{evw}, and indeed the proof here is a modification of the proof of that theorem.

The reader may note that \cite[Thm 8.8]{evw} requires not only that $q$ is not a multiple of $\ell$ but that $q$ is not congruent to $1$ modulo $\ell$.  We face no such restriction here.  That's because \cite[Thm 8.8]{evw} computes arbitrary moments of the Cohen-Lenstra distribution, whereas we are only studying the analogue of the average size of the $\ell$-part of the class group.  In the language of  \cite[Thm 8.8]{evw}, we are only considering the case $A = \Z/\ell\Z$.  The difference is as follows.  In the proof, we will end up estimating the number of $\F_q$-points on a moduli space over $\F_p$, and the result will depend on that space having just one geometrically irreducible component defined over $\F_q$.  In the more general setting treated in  \cite[Thm 8.8]{evw}, that space has many geometric components, all but one of which have fields of definition containing $\mu_\ell$; so when $q$ is congruent to $1$ mod $\ell$ there are multiple $\F_q$-rational components.  In the case treated here, the moduli space in question is geometrically irreducible, so this issue does not arise.

\begin{proof}
We begin by observing that $\sum_{f \in Q_{n,q}} m_a(f)$ can be interpreted as the number of $\F_q$-rational points of a certain moduli space.  

To this end we briefly recall the setup of  \cite[Section 7]{evw}. 

Let $k$ be a field, let $G$ be a finite group with trivial center, denote by $e$ the identity element of $G$, and let $c$ be a conjugacy-closed subset of $G \backslash e$.  By a {\em tame $G$-cover of $\P^1$ with monodromy type $c$} we mean a triple $(X,f,\phi)$ where
\begin{itemize}
\item  $X$ is a smooth proper geometrically connected curve $X/k$;
\item $f: X \ra \P^1$ is a tamely ramified finite cover;
\item The image of tame inertia at each branch point of $f$ excepting $\infty$ lies in $c$;
\item $f$ is Galois with group $G$; that is, $\Aut(f)$ acts transitively on the geometric fibers of $f$ and $\phi$ is an automorphism from $G$ to $\Aut(f)$.
\end{itemize}

Here by an isomorphism between two covers $f: X \ra \P^1$ and $f': X' \ra \P^1$ we mean a morphism $\psi: X \ra X^{'}$ with $f' \circ \psi = f$, not a pair $(\psi,\iota)$ with $\iota$ a nontrivial automorphism of $\P^1$ and $f' \circ \psi = \iota \circ f$.  In other words, our $\P^1$ is ``labeled''.

Then, as in \cite[Section 7]{evw} (more or less immediate from a theorem of Romagny and Wewers~\cite{romagnywewers}), there is a scheme $\Hn_{G,n}^c$ over $\Z[1/|G|]$ whose $k$-points (as long as $k$ has characteristic prime to $|G|$) are in bijection with the isomorphism classes of tame $G$-covers of $\P^1$ which have $n$ branch points on $\A^1$ with monodromy type $c$. (We do not specify whether or how the cover is branched at $\infty$.)\footnote{The somewhat artificial special treatment of $\infty$ in this definition, as in \cite{evw}, stems from the need to compare with topology, where branched covers of the disc are technically easier to handle than branched covers of the sphere.}  In fact  (\cite[Theorem 2.1]{romagnywewers}), for a scheme $S$, the set $\Hn_{G,n}^c(S)$ corresponds to isomorphism classes of tame $G$-covers over $S$, suitably defined; we will not need to spell out that definition here. Once the $n$ branch points are chosen on $\A^1$ there are finitely many choices for $f$ and $\phi$. Thus the dimension of $\Hn_{G,n}^c$ equals to $n$.

From now on, we suppose that $k$ is $\F_q$, that $G$ is the dihedral group  $\Z/\ell \Z \rtimes \Z/2\Z$, and that $c$ is the conjugacy class of an involution in $G$. We will now explain the relationship between the space of $G$-covers and the $\ell$-torsion in the Jacobian of  hyperelliptic curves.  The key point is that, for any algebraic curve $C$, the set of surjections $\Jac(C)[\ell] \ra (\Z/\ell\Z)^k$ is naturally identified with the set of \'{e}tale $(\Z/\ell\Z)^k$ covers of $C$. For details, see Section 3.9 of \cite{Milne}. 

If $f: X \ra \P^1$ is a $G$-cover, the product structure of $G$ allows us to factor $f$ as
\beq
X \stackrel{g}{\ra} C \stackrel{h}{\ra} \P^1
\eeq
where $h$ is a hyperelliptic cover and $g$ is a Galois cover with group $\Z/\ell\Z$; that is, $g$ is endowed with an isomorphism $\phi: \Z/\ell\Z \ra \Aut(g)$.  What's more, the fact that the monodromy in $f$ is of type $c$ implies that $g$ is an \'{e}tale cover, at least away from the points of $C$ over $\infty \in \P^1$.  

What happens over $\infty$ is slightly more delicate. The double cover $h$ is branched at $n$ points on $\A^1$, but the total number of branch points of $h$ must be even as $C$ is a smooth proper hyperelliptic curve.  Thus, if $n$ is odd, $h$ is branched at $\infty$.  The monodromy around $\infty$ in the cover $X \ra \P^1$ is thus an element of $G$ projecting to the nontrivial element of $\Z/2\Z$.  Such an element must be an involution, and it follows that $g$ is unramified at $\infty$.  If $n$ is even, on the other hand, it is possible for $g$ to be ramified.  We thus wish to restrict our attention to those $G$-covers $X \ra C$ which are unramified over $\infty$.  These are parametrized by a closed and open subscheme of $\Hn_{G,n}^c$ (indeed, it is the second term in the disjoint union in the paragraph following (7.3.1) of \cite{evw}). Let $X_n$ be this subscheme of $\Hn_{G,n}^c$ when $n$ is even, and $\Hn_{G,n}^c$ when $n$ is odd. In both cases, $\dim X_n = n$. We have explained how every point of $X_n(k)$ gives rise to a triple $(g,\phi,h)/k$ up to isomorphism, and in fact it is not hard to check that the converse holds as well.  (This is essentially the last paragraph of the proof of \cite[Proposition 8.7]{evw}.) 

If $a$ is an element of $(\Z/\ell\Z)^*$, we denote by $\a$ the automorphism of $X_n$ which sends $(g,\phi,h)$ to $(g,a\phi,h)$.  We then write $X_n^a$ for the twist of $X_n$ by the homomorphism $$\Gal(\Fqbar/\F_q) \ra \Aut(X_n)$$ which sends $\Frob_q$ to $a$.  (Reference:  \cite[\S 4.5]{poonen:rp}.)  

\begin{lemma} With notation as in Section \ref{Notation}, 
\beq
\sum_{f \in Q_{n,q}} m_a(f) = (q-1)|X_n^a(\F_q)|
\eeq
\end{lemma}

\begin{proof}

A point of $X_n^a(\F_q)$ is a point of  $X_n(\Fqbar)$ such that $\Frob_q \cdot x = \a \cdot x$.  In other words, it is a triple $(g,\phi,h)/\Fqbar$ such that $\Frob_q \cdot (g,\phi,h)$ is isomorphic to $(g,a\phi,h)$.  The fact that the isomorphism class of $h$ is fixed by Frobenius implies that the branch locus of $h$ is an $\F_q$-rational divisor. Let $f(x) \in \mathbb{F}_q[x]$ be the unique monic squarefree polynomial which vanishes precisely at the branch locus of $h$. Then $C$ is isomorphic (over $\Fqbar$) to the smooth completion of the hyperelliptic curve defined by $y^2 = f(x)$.

Fixing such an $h$,  and thus such a $C$, we now consider the set of points of $X_n^a(\F_q)$ lying over this $h$.  First of all, the choices of $(g,
\phi)$ such that $(g,\phi,h) \in X_n^a(\Fqbar) = X_n(\Fqbar)$ for a specified $h$ are in bijection with the $\ell^{2g(C)}-1$ surjections from $J(C)[\ell](\Fqbar)$ to $\Z/\ell\Z$.  Two such surjections $s,s'$ are isomorphic (that is, are parametrized by the same point of $X_n^a(\Fqbar)$) if and only if $s = \pm s'$.  The action of Frobenius on the set of surjections sends $s$ to $a^{-1} \Frob_q s$; so $s$ descends to a point of $X_n^a(\F_q)$ if and only if $\Frob_q \cdot s = \pm a s$.  We conclude that the number of points of $X_n^a(\F_q)$ lying over $h$ is $(1/2)(m_a(f) + m_{-a}(f))$. 

Now if $f$ in $Q_{n,q}$ is {\em not} monic then $f = \epsilon F$ for some $\epsilon \in \F_q^*$ and some monic $F$.  The curve $C_f$ is isomorphic to $C_F$ if $\epsilon$ is a quadratic residue and to the nontrivial quadratic twist of $C_F$ otherwise.  In the former case, $m_a(f) = m_a(F)$, and in the latter, $m_a(f) = m_{-a}(F)$.  In particular, the quantity $(1/2)(m_a(f) + m_{-a}(f))$ is the same for all $q-1$ nonzero multiples of $F$.  We conclude that

\beq
\sum_{f \in Q_{n,q}}(1/2)(m_a(f) + m_{-a}(f)) = (q-1)|X_n^a(\F_q)|
\eeq

Moreover, taking $\epsilon$ to be a non-residue in $\F_q^*$,
\beq
\sum_{f \in Q_{n,q}} m_a(f) = \sum_{f \in Q_{n,q}} m_a(\epsilon f) = \sum_{f \in Q_{n,q}} m_{-a}(f)
\eeq
from which we obtain
\beq
\sum_{f \in Q_{n,q}} m_a(f) = (q-1)|X_n^a(\F_q)| 
\eeq
as desired.
\end{proof}

We now argue exactly as in the proof of \cite[Theorem 8.8]{evw}.

Since $|Q_{n,q}| = (q-1)(q^n-q^{n-1})$, it suffices to prove that
\begin{equation}
|q^{-n}|X_n^a(\F_q)| - 1| \leq C_\ell q^{-1/2}
\label{eq:xnform}
\end{equation}

for some $C_\ell$ depending only on $\ell$ and for all $n > N_\ell, q > Q_\ell$. 

Via the Grothendieck-Lefschetz trace formula, we have
\begin{equation}
|X_n^a(\F_q)| = \sum_i (-1)^i \Tr(\Frob_q | H^i_{c,\text{\'et}}((X_n^a)_{\Fqbar}; \Q_\lambda).
\label{eq:gl}
\end{equation}

where $\lambda$ is a prime greater than $\max \{2\ell,q,n\}$.  

Note that the \'etale cohomology is that of the base change of $X_n^a$ to $\Fqbar$, where it becomes isomorphic to the untwisted space $X_n$; in particular, the choice of $a$ affects the action of Frobenius on the \'etale cohomology, but not the \'etale Betti numbers, bounds on which are the main engine of the argument.

We begin by computing the main term:
\beq
\Tr(\Frob_q | H^{2n}_{\text{c,\'et}}((X_n^a)_{\Fqbar}; \Q_\lambda) = q^n.
\eeq
This follows immediately from the fact that $(X_n^a)_{\Fqbar} \cong (X_n)_{\Fqbar}$ is irreducible.  When $n$ is odd, this is shown in the proof of \cite[Theorem 8.8]{evw} as a consequence of a big monodromy theorem of J.K. Yu.  (This is actually the only place where we need $n$ to be large, and indeed $n=3$ would be enough.)  When $n$ is even, we argue as follows.  The map from $X_n$ to the configuration space $\Conf^n \A^1$ sending a $G$-cover to its branch locus is a finite cover (\cite[Section 2.2]{evw}), and irreducibility of $X_n$ is equivalent to the monodromy group of this cover acting transitively on the fiber.  It suffices to check that this holds on a closed subvariety of the base.  So write $Z$ for the subvariety of $\Conf^n \A^1$ consisting of those configurations containing some specified point $p_0 \in \P^1(F_q)$, and let $Y$ be the preimage of $Z$ in $X_n$.  An automorphism of $\P^1$ taking $p_0$ to $\infty$ now identifies $Y$ with $\Hn_{G,n-1}^c$, which we know to be irreducible since $n-1$ is odd.  This implies that $X_n$ is irreducible.

We now turn to the error term.
  The moduli space $X_n$ is a closed and open subscheme of $\Hn_{G,n}^c$, so its Betti numbers are bounded by those of $\Hn_{G,n}^c$; by \cite[(7.8.1)]{evw} we have
\beq
\dim H^{2n-i}_{\text{c,\'et}}((X_n^a)_{\Fqbar}; \Q_\lambda) \leq K_\ell (B_\ell)^i
\eeq
where $K_\ell,B_\ell$ are constants depending only on $\ell$. 

Using the Deligne bound \cite{Deligne}, the eigenvalue of Frobenius on $H^{i}_{c,\text{\'et}}((X_n^a)_{\Fqbar}; \Q_\lambda)$ is bounded in absolute value by $q^{i/2}$; so the absolute value of the contribution of all $i < 2n$ to \eqref{eq:gl} is bounded above by the sum of a geometric series which converges for all $q > B_\ell^2$. In particular, as in \cite[Section 1.8]{evw}, this contribution is at most
\beq
K_\ell B_\ell q^{-1/2}(1-B_\ell q^{-1/2})^{-1}q^n.
\eeq
So if we take $Q_\ell = 4 B_\ell^2$ and $q > Q_\ell$, we may take $C_\ell = 2 K_\ell B_\ell$ and conclude
\beq
|X_n^a(\F_q) - q^n|  = |\sum_{i=0}^{2n-1} (-1)^i \Tr(\Frob_q | H^i_{c,\text{\'et}}((X_n^a)_{\Fqbar}; \Q_\lambda)| < C_\ell q^{n-1/2}
\eeq
which proves \eqref{eq:xnform} and thus the desired result.

\end{proof}

Proposition~\ref{prop:twistedhurwitz} allows us to bound the proportion of hyperelliptic curves whose \'{e}tale cohomology has a Frobenius eigenvalue congruent to $a$ mod $\ell$.  Recall from Section~\ref{Notation} that  $Q_{n,q}^{a,\ell}$ is the set of squarefree polynomials over $\F_q$ of degree $n$ such that $m_a(f)$ is greater than $0$.

\begin{corollary} \label{MainCoro}
There are constants $C'_\ell,Q_\ell,N_\ell$ such that for any
$a \in \left( \Z / \ell \Z \right)^\times$ we have
\beq
\frac{|Q_{n,q}^{a,\ell}|}{|Q_{n,q}|} \leq \frac{1}{\ell -1}+C'_{\ell} q^{-1/2}
\eeq
for all $n \geq N_\ell,q \geq Q_\ell$.
\end{corollary}

\begin{proof}
Write $\delta$ for the quantity $|Q_{n,q}|^{-1} |Q_{n,q}^{a,\ell}|$ to be bounded.

Since $m_a(f)$ is the number of nonzero elements of a vector space over $\Z/\ell\Z$, it is at least $\ell-1$ if it is greater than $0$.  In particular,
\beq
|Q_{n,q}|^{-1} \sum_{f \in Q_{n,q}} m_a(f) \geq |Q_{n,q}|^{-1} (\ell-1)|Q_{n,q}^{a,\ell}| =(\ell-1)\delta
\eeq
By Proposition~\ref{prop:twistedhurwitz}, we now have
\beq
(\ell - 1)\delta < 1+ C_\ell q^{-1/2}
\eeq
for all sufficiently large $n,q$, which yields the desired result by taking $C'_\ell = C_\ell / (\ell-1)$.
\end{proof}

So far, we have used the results of \cite{evw} as they appear in that paper.  However, for the present application, it is useful to compute explicit values $C_\ell,Q_\ell$ for which Corollary~\ref{MainCoro} holds.  We do so by going back to the main body of \cite{evw} and working out explicit bounds for quantities that are given in \cite[(7.8.1)]{evw} only as unspecified constants.

\begin{proposition} Corollary~\ref{MainCoro} holds with $C'_\ell =2 (\ell - 1)^{-1} (4 \ell)^{138}$ and $Q_\ell = 4 \cdot (4 \ell)^{156}$.
\label{prop:explicitbounds}
\end{proposition}

\begin{proof}  By the proof of Corollary~\ref{MainCoro}, we may take $C'_\ell$ to be $C_\ell / (\ell - 1)$, where $C_\ell$ is the constant appearing in the statement of Proposition~\ref{prop:twistedhurwitz}. Moreover, we may take $C_\ell$ to be $2K_\ell B_\ell$ and $Q_\ell$ to be $4 B_\ell^2$, where $B_\ell,K_\ell$ are the constants appearing in the proof of Proposition~\ref{prop:twistedhurwitz} controlling the exponential growth of the Betti numbers of the relevant Hurwitz space.  We now explain how to bound $B_\ell$ explicitly.

In \cite[7.8.1]{evw}, the bound
\beq
\dim H^{i}_{\text{\'et}}((X_n^a)_{\Fqbar}; \Q_\lambda) \leq K_\ell(B_\ell)^i
\eeq
arises from two facts. First, there is a stability theorem~\cite[6.2]{evw}, which tells us in this context that
\begin{equation}
\dim H^{i}_{\text{\'et}}((X_n^a)_{\Fqbar}; \Q_\lambda) = 
\dim H^{i}_{\text{\'et}}((X_{n+D}^a)_{\Fqbar}; \Q_\lambda)
\label{eq:stab}
\end{equation}
for all $n > Ai+B$, where $A,B,$ and $D$ are constants we shall specify. Second, there is an absolute bound~\cite[2.5]{evw} which tells us that
\beq
\dim H^{i}_{\text{\'et}}((X_n^a)_{\Fqbar}; \Q_\lambda) \leq (4\ell)^n.
\eeq
These two facts together imply that
\beq
\dim H^{i}_{\text{\'et}}((X_n^a)_{\Fqbar}; \Q_\lambda) \leq (4\ell)^{Ai+B+D}
\eeq
so we may take $B_\ell = (4\ell)^A$ and $K_\ell = (4\ell)^{B+D}$.  It remains to compute $A,B,$ and $D$. 

The key object of computation is the ring $R$ defined in \cite[\S 3]{evw}.  This ring is defined for any finite group $G$ and any conjugacy-closed subset of $G$; we will consider here just the case relevant to us, which is that where $G$ is the dihedral group of order $2\ell$ and $c$ is the class of involutions in $G$.  The set of $n$-tuples of involutions $(\tau_1, \ldots, \tau_n) \in G^n$ carries a natural action of the $n$-strand braid group; the ring $R$ is a graded $\Q$-algebra whose degree-$n$ part is spanned by the set of orbits of that action, which set we denote $\Sigma_n$.  The multiplication in $R$ is given by concatenation of $n$-tuples.

The key fact about $R$ is that it contains a central element $U$ with the property that $R[U]$ and $R/UR$ both have finite degree (that is, they are supported in only finitely many grades.)  In the dihedral case, $R$ and $U$ are particularly easy to describe.  For any $n$, there is map from $\Sigma_n$ to $G$ sending $(\tau_1, \ldots, \tau_n)$ to the product $\tau_1 \ldots, \tau_n$, which is called the {\em boundary monodromy}. Each $n$-tuple in $\Sigma_n$ also has a {\em monodromy group}; namely, the group generated by $\tau_1, \ldots, \tau_n$.  The possible monodromy groups are just the order-$2$ subgroups of $G$ and $G$ itself.  It is not hard to check that, for all $n \geq 4$, the elements of $\Sigma_n$ are determined by their boundary monodromy and their monodromy group; to be precise, $\Sigma_n$ consists of $\ell$ orbits consisting of the single element $(\tau, \tau, \ldots, \tau)$ as $\tau$ ranges over the $\ell$ involutions, and $\ell$ more orbits, each of which consists of all $n$-tuples with monodromy group $G$ and boundary monodromy $g$, as $g$ ranges over the index-$2$ cyclic subgroup of $G$ (when $n$ is even) or its nontrivial coset (when $n$ is odd.)  In particular, $\dim R_n = 2\ell$ for all $n \geq 4$.  We may take $U$ to be the degree-$2$ central operator 
\beq
U = \sum_{\tau \in c} (\tau,\tau)
\eeq
and check that $U$ induces an isomorphism from $R_n$ to $R_{n+2}$ for all $n \geq 4$.  In particular, $\deg R[U]$ and $\deg R/UR$ are both at most $4$, where by the degree of a graded ring we mean the highest grade represented in its support.

This combinatorial information about the dihedral group is what goes into the computation of constants in \cite{evw}.  The constant $D$ in \cite[6.1]{evw} is just the degree of $U$, which is $2$.  The stability result in \cite[6.1]{evw} is derived from a general theorem~\cite[4.2]{evw} about $R$-modules.  The $R$-module $M$ governing the $H^i$ of Hurwitz space, to which we apply \cite[4.2]{evw} is the one called $M_i$ in \cite[6.1]{evw},  So (using the constants appearing in those theorems) stability begins when $n = \max(h_0,h_1) + A_0$, where $h_j$ is the quantity denoted $\deg H_j(\mathcal{K}(M_i)$ in \cite[6.1]{evw}.  In turn, as asserted in the first paragraph of the proof of \cite[6.1]{evw}, we have
\beq
\deg H_j(\mathcal{K}(M_i)) \leq A_2 + A_0(3i + j)
\eeq
So we find that \eqref{eq:stab} holds for all $n \leq A_2 + A_0(3i+1) + A_0 = 3A_0 i + (2A_0 + A_2).$  In other words, we may take $A = 3A_0$ and $B = 2A_0 + A_2$.

Finally, the values of $A_0$ and $A_2$ are given in \cite[4.5.3]{evw}.  They are defined in terms of $A(R) = \max(\deg R[U],\deg R/UR)$, which for us is $4$.  Now $A_0 = 6A(R) + \deg U = 26$ and $A_2 = A(R) + \deg U = 6$.  Thus, $A = 78$ and $B=58$.  Since $D=2$, we conclude that we may take $B_\ell = (4 \ell)^{78}$ and $K_\ell = (4 \ell)^{60}$.  So we have $Q_\ell = 4 \cdot (4 \ell)^{156}$ and $C'_\ell = 2 (\ell - 1)^{-1} (4 \ell)^{138}$, as claimed.

\end{proof}

\section{Application to nonvanishing of $L$-functions}

We can use the above reasoning to bound the number of quadratic $L$-functions over function fields which vanish at a specified point on the critical line.  For the rest of this section we fix an odd prime $p$ and consider only fields of characteristic $p$. We note that, if $\chi_f$ is a quadratic character of $\F_q(x)$, then $L(s,\chi_f)$ can vanish only at a point $s$ such that $q^s$ is a $q$-Weil number of weight $1$. We first recall the following lemma relating the vanishing of the $L$-function of a quadratic character in terms of the Frobenius eigenvalues of a hyperelliptic curve; 

	\begin{lemma}\label{charactertocurve}
	Let $f$ be a monic squarefree polynomial in $\mathbb{F}_q[x]$ and $\chi_f$ be the quadratic character with conductor $f$. Let $C$ be the hyperelliptic curve defined by $y^2 = f(x)$ and let $P \in \mathbb{Z}[x]$ be the characteristic polynomial of geometric Frobenius acting on the Jacobian of $C$.  Then for any $s \ne 0$,  $L(s,\chi_f) = 0$ if and only if $P(q^s) = 0$.
	\end{lemma}

This is immediate from the description of $P$ as the numerator of the zeta function of $C$, and the connection of the latter to $L(s, \chi_f)$ (see, for instance, \cite[Section 2]{R}).

\begin{theorem}\label{ThmNonvanishing}
 For any squarefree polynomial $f \in Q_{n,q}$, let $L(s,\chi_f)$ be the Dirichlet $L$-function associated to the quadratic character $\chi_f$ as was defined in Section \ref{Notation}. Then for any $s \ne 0$,
$$ \limsup_{n \to \infty} \frac{|\{f \in Q_{n,q} \mid L(s, \chi_f) = 0  \}|}{|Q_{n,q}|} \ll  q^{-1/276} $$
where the limit is taken over all powers $q$ of a fixed odd prime number $p$.
\end{theorem}

\begin{proof}
    Fix an odd prime number $p$, and let $q$ be a power of $p$.
	By Lemma \ref{charactertocurve}, $L(s,\chi_f)=0$ is equivalent to $P(q^{-s})=0$ where $P(x)\in \Z[x]$ is the characteristic polynomial of Frobenius acting on the Jacobian of the hyperelliptic curve defined by $y^2=f(x)$. Thus, the set $\{f \in Q_{n,q} \mid L(s, \chi_f) = 0  \}$ is the same as $Q_{n,q}^{q^s}$. 	
	
	By Chebotarev's density theorem, we can (for large enough $q$) find a prime 
	$$ \ell = \frac{1}{4} \left(\frac{q}{4}\right)^{1/276}(1 + o(1)) $$
	mod which $g_{p^s}$, the minimal polynomial of $p^s$, splits completely. Let $a \in \mathbb{Z}/l \mathbb{Z}$ such that $g_{p^s}(a) = 0 \bmod \ell$.  If $q = p^t$, then any $f$ with $L(\chi_f,s) = 0$ has $m_{a^t}(f) > 0$.  So
\beq
\frac{|Q_{n,q}^{q^s}|}{|Q_{n,q}|} \le \frac{|Q_{n,q}^{a^t,\ell}|}{|Q_{n,q}|}
\eeq
and now we can apply Corollary \ref{MainCoro} to conclude using the second equation of \ref{FrobActEq}
that, for all sufficiently large $t$, we have
\beq
\limsup_{n \to \infty} \frac{|Q_{n,q}^{q^s}|}{|Q_{n,q}|}  \leq \frac{1}{\ell -1}+C'_{\ell} q^{-1/2}.
\eeq
The required bound follows from Proposition \ref{prop:explicitbounds}.

\end{proof}

Results on the vanishing of quadratic $L$-functions over function fields can be used to study the rank distribution of quadratic twist families of constant abelian varieties. In the following corollary, we show that as the constant field grows (so the characteristic is not changing), the probability for a quadratic twist of a constant abelian variety to have positive rank goes to $0$.  In the elliptic curve case, this agrees with the general  ``Minimalist Conjecture" philosophy, which holds that positive ranks should be a density-$0$ phenomenon except when forced by parity considerations from the functional equation (in this setting the functional equation never forces positive rank, and the rank is always even.)

	\begin{corollary}\label{rank0}
	Let $A$ be an abelian variety defined over a finite field $\mathbb{F}_q$ of odd characteristic. For each $f \in Q_{n,q^m}$, denote by $A_f$ the quadratic twist of $A \times_{\F_{q}} \F_{q^m}(x)$ by $f$.
	Let $R_{n,m}$ be the set $\{ f \in Q_{n,q^m}: A_f \text{ has positive rank} \}$.
	Then $$\lim\limits_{m \to \infty} \limsup_{n \to \infty} \frac{|R_{n,m}|}{|Q_{n,q^m}|} =0 .$$
	\end{corollary}

	\begin{proof}
		Let $P(x)$ be the characteristic polynomial of Frobenius acting on the Tate module of $A$ and let $q^{-s}$ be one of its roots. Then $\rank A_f >0$ is equivalent to $L(s,\chi_f)=0$. (See \cite[Proposition 4.6]{Wanlin} for a similar statement with the same proof.) Thus, the statement is a direct application of Theorem \ref{ThmNonvanishing}.
	\end{proof}

We now prove Theorem~\ref{SpecialCasefromMatrix}, which makes use of the mod $2$ Galois representations on $J(C)$ rather than the representations modulo odd primes.

\begin{proof}[Proof of Theorem \ref{SpecialCasefromMatrix}]

	Let $x_1,...,x_{2g+2}$ be the set of Weierstrass points of $C$.  The $2$-torsion subgroup $J(C)[2]$ is spanned by the degree-$0$ $2$-torsion divisors $x_i - x_j$.  That is, the group of divisors of the form $\sum a_i x_i$ with $\sum a_i = 0$ surjects onto $J(C)[2]$.  Note also that $x_1 + \ldots + x_{2g+2}-(2g+2)x_1$ is a principal divisor and thus $x_1 + \ldots + x_{2g+2}$ is $0$ in $J(C)[2]$. See \cite[Secion 4]{Gross} for detailed discussion.  This identifies $J(C)[2]$ with an explicit subquotient of $\F_2^{2g+2}$; namely, $J(C)[2]$ is the quotient of the subspace $(a_1, \ldots, a_{2g+2}): \sum a_i = 0$ by the $1$-dimensional subspace spanned by $(1,\ldots,1)$.  
	
	This identification is equivariant for the Frobenius action on both sides, so it allows us to describe the mod $2$ Galois representation afforded by $J(C)$ in terms of the permutation $\pi$ which Frobenius induces on $x_1, \ldots, x_{2g+2}.$  To be precise, the action of $S_{2g+2}$ on $J(C)[2]$ is a representation $\rho: S_{2g+2} \ra \Sp_{2g}(\Z/2\Z)$, and the action of Frobenius on $J(C)[2]$ is given by $\rho(\pi)$. 
	
	The conditions on $\pi$ given in Theorem~\ref{SpecialCasefromMatrix} are equivalent to the condition that $\pi^2$ is a product of two disjoint odd cycles.  Thus, the action of $\pi^2$ in its permutation representation $\F_2^{2g+2}$ has eigenvalues given by $\mu_k$ and $\mu_{2g+2-k}$ for some odd $1 \leq k \leq 2g+1$; passing to the subquotient $J(C)[2]$ removes two eigenspaces of $\rho(\pi^2)$ with the eigenvalue $1$. So the eigenvalues of $\Frob^2$ on $J(C)[2]$ are the multiset $\mu_{k}'  \cup \mu_{2g+2-k}'$ where $\mu_n'$ denotes the nontrivial $n$th roots of unity.  We see in particular that $\rho(\pi^2)$ does not have $1$ as an eigenvalue.  But if the zeta function $Z_C$ had a zero at $1/2$, then $\sqrt{q}$ would be a Frobenius eigenvalue on $C$, which would mean that $q$ was an eigenvalue of $\Frob^2$; we have shown that $\Frob^2$ has no eigenvalue congruent to $1$ mod $2$, which rules this out.  This proves (1).  
	
	What's more, the multiset $\mu_k' \cup \mu_{2g+2-k}'$ contains any eigenvalue at most twice, and if $(k,2g+2-k) = 1$, no eigenvalue appears more than once.  This proves (2) (or rather, it proves (2) for the zeta function of $C/\F_{q^2}$, from which (2) is immediate.)

\end{proof}

\begin{proof}[Proof of Corollary \ref{quadcharacterCorollary}]

By assumption, $f$ is an irreducible polynomial over $\mathbb{F}_q$. So when $\deg f = n$ is even, Frobenius acts on the set of Weierstrass points of $C_f: y^2 = f(x)$ as a $n$-cycle. If $\deg f = n$ is odd, then Frobenius acts on the set of Weierstrass points of $C_f$ as a disjoint union of a $n$-cycle and a $1$-cycle. In either case, when $4 \nmid \deg f$ we can apply Theorem \ref{SpecialCasefromMatrix} to conclude that $Z_{C_f}$ does not vanish at $1/2$ and so behaves $L(s,\chi_f)$.

 For any $X=q^{2g+2}$, the set of irreducible polynomials of odd degree at most $2g+1$ gives quadratic characters with bounded conductor whose $L$-function does not vanish at the central point $s=1/2$. By the prime number theorem for function fields, the number of irreducible polynomials in $\mathbb{F}_q[x]$ of degree at most $n$ is $\gg q^n/n$. This gives the lower bound in statement.

\end{proof}

The proof of Theorem~\ref{ModpVersion} is very similar to that of Theorem \ref{SpecialCasefromMatrix}, but we treat it separately in order to make the hyperelliptic case above more readable.

\begin{proof}[Proof of Theorem \ref{ModpVersion}]

Let $x_1, \ldots, x_m$ be the ramification points of the $(\Z/\ell\Z)$-cover of $\P^1$ in $S$, where $m=k_1 + \dots + k_r$. The Jacobian $J(C)$ of $C$ carries an action of $\Z[(\Z/\ell\Z)]$; write $\lambda \in \Z[(\Z/\ell \Z)]$ for $\zeta_\ell-1$,  where $\zeta_\ell$ is a generator of $(\Z/\ell \Z)$.  A Riemann-Hurwitz computation shows that the genus of $C$ is $(m-2)(\ell-1)/2$, so the Tate module $T_\ell J(C)$ is a free $\Z_\ell[\zeta_\ell]$-module of rank $m-2$, and $J(C)[\lambda]$ has dimension $m-2$. 

The $\lambda$-torsion subgroup of $J(C)$ is spanned by the degree-$0$ $\lambda$-torsion divisors $x_i - x_j$.  That is, the group of divisors of the form $\sum a_i x_i$ with $\sum a_i = 0$ surjects onto $J(C)[\lambda]$.  This surjection is not an isomorphism; there is a $1$-dimensional kernel, which we can describe as follows.  Over $\Fqbar$, the curve $C$ has an affine model of the form $y^\ell = f(x)$ with $f$ a rational function with no zeroes or poles at $\infty$.  Then the principal divisor associated to $y$ is $\sum a_i x_i$ where $a_i = \ord_{x_i} f$.  We have now expressed $J(C)[\lambda]$ as an explicit subquotient of $\F_\ell^m$.

This identification is equivariant for the Frobenius action on both sides, so it allows us to describe the mod $\ell$ Galois representation afforded by $J(C)$ in terms of the permutation $\pi$ which Frobenius induces on $x_1, \ldots, x_m.$ 

The action of $\pi$ splits $x_1, \ldots, x_m$ into cycles of length $k_1, \ldots, k_r$, which by hypothesis are prime to $\ell$, and which must be multiples of $d$, where $d$ is the order of $q$ in $\F_\ell^*$.  So the eigenvalues of $\pi$ in its action on $\F_\ell^m$ are the union (as multisets) $\bigcup_{j=1}^r \mu_{k_j}$.  Now the composition factors of $\F_\ell^m$ as a representation of the cyclic group $\langle \pi \rangle$ are $J(C)[\lambda]$, $\F_\ell \mbox{div}(y)$, and the $\pi$-trivial one-dimensional representation onto which $\F_\ell^m$ maps by summing coordinates. The action of $\pi$ on the latter factor is trivial, while $\pi$ acts on $\F_\ell \mbox{div}(y)$ as multiplication by $q$.  If the zeta function $Z_C$ had a zero at $1/2$, then $\sqrt{q}$ would be a Frobenius eigenvalue of $C$, which would mean that some eigenvalue $\mu$ of the action of $\pi$ on $J(C)[\lambda]$ satisfied $\mu^2 = q$.

In case $q$ is congruent to $1$ modulo $\ell$ (i.e., $d=1$) the eigenvalues of $\pi$ in its action on $J(C)[\lambda]$ are the multiset $\bigcup_{j=1}^r \mu'_{k_j}$ together with $r-2$ copies of $1$, where $\mu_n'$ denotes the nontrivial $n$th roots of unity.  The hypotheses $r=2$ and $k_j$ odd now guarantee that the eigenvalues of $\pi$ on $J(C)[\lambda]$ contain no copies of either $1$ or $-1$, completing the proof in this case.

If $d>1$, we note that our condition on $q^{\kappa_j}$ can be satisfied only when $d$ is even.  (If $d$ is odd, then $\kappa_j$ is always a multiple of $d$, so $q^{\kappa_j}= 1$.)  When $d$ is even, our condition in fact says precisely that each $k_i$ is a multiple of $d$ but not of $2d$.  The two square roots of $q$ in $\bar{\F}_\ell^*$ both have order $2d$, and thus neither can appear among the eigenvalues of Frobenius on $J(C)[\lambda]$.

We now turn to the first assertion of the theorem.  We note that the coprimality of the $k_i$ implies that $d=1$, or in other words that $q$ is conguent to $1$ modulo $\ell$.  The fact that $J(C)$ carries an action of $\Z/\ell\Z$ defined over $\F_q$ with trivial invariant subspace implies that the Frobenius eigenvalues on $J(C)$ all appear with multiplicity a multiple of $\ell-1$, and that a root of multiplicity $k(\ell-1)$ reduces to a root of multiplicity $k$ in the action of Frobenius on $J(C)[\lambda]$.  So we just need to show that the action of Frobenius on $J(C)[\lambda]$ has no repeated eigenvalues.  The coprimality of the $k_i$ guarantees that the union $\bigcup_{j=1}^r \mu'_{k_j}$ is disjoint; the remaining eigenvalues are $r-2$ copies of $1$, so since $r \leq 3$ we are done.

\end{proof}
		
\begin{proof}[Proof of Corollary \ref{characterCorollary}]
We first consider the case $d=1$ (i.e., $q$ is congruent to $1$ modulo $\ell$).

Let $f_1,f_2$ be two distinct monic irreducible polynomials in $\mathbb{F}_q[t]$ of degrees being odd and prime to $\ell$. Let $d_1 = \deg f_1$ and $d_2 = \deg f_2$. Let $e \in \{1,\ldots,\ell-1\}$ be such that $\ell \mid d_1+ed_2$. The smooth projective curve $C$ with affine model $y^\ell =f_1f_2^e$ admits a cyclic degree $\ell$ map to $\mathbb{P}^1_{\mathbb{F}_q}$ where the set of branch points are roots of $f =f_1f_2$ in the affine line (the condition on $d_1 + ed_2$ guarantees there is no branching at $\infty$.) By Theorem 1.7(2), the zeta function $Z_C$ does not vanish at $s=1/2$. Thus, all degree $\ell$ Dirichlet characters with conductor $f$ have $L$-functions nonvanishing at the central point.  Evidently, the number of such pairs $f_1,f_2$ with $d_1 + d_2 \le n$ is at least the number of irreducible polynomials of degree $n-3$; the number of characters with non-vanishing L-functions is thus bounded below by a constant multiple of $q^n/n = X/(\log X)$.  This concludes the proof in the $d=1$ case.

We now turn to the case where $q \not\equiv 1 \bmod \ell$. We recall that $d$ is the order of $q \bmod \ell$ in $(\mathbb{Z}/\ell\mathbb{Z})^*$. Let $\Sigma_{d,2d}$ be the set of monic squarefree polynomials such that all of their irreducible factors have degree divisible by $d$ but not divisible by $2d$. For any $f \in \Sigma_{d,2d}$ there exists a $(\mathbb{Z}/\ell\mathbb{Z})$ field extension $K/\mathbb{F}_q(x)$ with conductor $f$. The field $K$ is the function field of a curve $C/\mathbb{F}_q$ which admits a cyclic degree $\ell$ map to $\mathbb{P}^1_{\mathbb{F}_q}$ whose set of branch points are roots of $f$. By Theorem 1.7(2), the zeta function $Z_C$ does not vanish at $s=1/2$. Thus, all degree $\ell$ Dirichlet characters with conductor $f$ have their $L$-functions do not vanish at the central point.  The number of such characters is $(\ell-1)^{\omega(f)}$ where $\omega(f)$ denotes the number of irreducible factors of $f$.

So it remains to count the number of degree $\ell$ Dirichlet characters whose conductor is in $\Sigma_{d,2d}$, which by the above discussion is given by
\[ \sum_{f \in \Sigma_{d,2d}, |f| \le X} (\ell-1)^{\omega(f)}. \]

To estimate this sum, we start by considering the Dirichlet series 
\[ G(s) = \sum_{f \in \Sigma_{d,2d}} (\ell-1)^{\omega(f)}|f|^{-s}  \]
where $|f| = q^{\deg f}$. Then $G(s)$ has a Euler product expansion
\[G(s) = \prod_{P \in \Sigma_{d,2d}, \text{ irrd.} }(1+(\ell-1)|P|^{-s}).\] Taking  
\[H(s) = \prod_{P \in \Sigma_{d,2d}, \text{ irrd.} } (1+|P|^{-s})^{2d}, \]
we see that $G(s)^{2d}/H(s)^{\ell-1}$ is holomorphic at $s=1$. Now define
$$Z_d(s) = \prod_{d \mid \deg P,\ P \text{ irrd.}}(1+|P|^{-s})^d $$
so that $H(s) = (Z_d(s))^2/Z_{2d}(s)$. Since $Z_d(s)$ and $Z_{2d}(s)$ each have a simple pole at $s=1$, so does $H(s)$.

We conclude that $G(s)^{2d}$ has a pole at $s=1$ with order $\ell-1$ and is absolutely convergent for $\Re(s) >1$. By \cite[Theorem 3.1]{Kato}, the sum of coefficients of $G(s)$ has the following asymptotic relation
\[ \sum_{f \in \Sigma_{d,2d}, |f| \le X} (\ell-1)^{\omega(f)} \sim C\cdot X(\log X)^{\frac{\ell-1}{2d}-1} \]
where the constant
\[ C = \frac{|(\lim_{s \to 1}G(s)^{2d}(s-1)^{\ell-1})^{\frac{1}{2d}}|}{\Gamma((\ell-1)/(2d))}. \]
And it gives the desired result.

\end{proof}

\bibliographystyle{amsplain}
\bibliography{ELSbib}

\end{document}